\documentclass[10pt]{article}

\usepackage[margin=2cm]{geometry}
\usepackage{amsmath,amsthm,amsfonts,cite}
\usepackage{thmtools}
\usepackage[yyyymmdd,hhmmss]{datetime}
\usepackage{tikz}

\declaretheorem[style=plain]{theorem}
\declaretheorem[sibling=theorem,style=plain]{corollary}
\declaretheorem[sibling=theorem,style=plain]{lemma}
\declaretheorem[sibling=theorem,style=plain]{proposition}

\declaretheorem[sibling=theorem,style=definition,qed=$\bullet$]{definition}

\declaretheorem[sibling=theorem,style=remark,qed=$\bullet$]{remark}

\allowdisplaybreaks[4]

\let\oldbibliography\thebibliography
\renewcommand{\thebibliography}[1]{%
  \oldbibliography{#1}%
  \setlength{\itemsep}{-2pt}%
}
\usepackage{hyperref}
\usepackage{float}
\usepackage{multirow}
\usepackage{tabu}
\usetikzlibrary{math,cd,decorations.shapes}

\newcommand{\eqn}[1]{\begin{equation*} #1 \end{equation*}}
\newcommand{\neqn}[1]{\begin{equation} #1 \end{equation}}

\newcommand{\set}[1]{\left\{ #1 \right\}}
\newcommand{\seqnum}[1]{\href{https://oeis.org/#1}{\rm \underline{#1}}}
\newcommand{\doi}[1]{\textsc{DOI}: \href{https://dx.doi.org/#1}{#1}}

\newcommand{\qquadtext}[1]{\qquad \text{#1} \qquad}
\newcommand{\seq}[0]{\operatorname{SEQ}}

\newcommand{\abs}[1]{\left\vert #1 \right\vert}

\newcommand{\latticepath}[8]{\resizebox{#1}{!}{
\begin{tikzpicture}
#7
\tikzmath{\a = #3; \b = #4; \n = #5; \an = \a*\n; \bn = \b*\n; \abn = \an+\bn; \nm = \n-1; \abnh = \abn-0.5;}
\foreach \i in {0,...,\an}
    \foreach \j in {0,...,\bn}
        \fill (\i,\j) circle (3pt);
\foreach \i in {1,...,\nm}
    \fill[white] (\a*\i,\b*\i) circle (4pt);
\fill (0,0) circle (3pt);
\fill (\an,\bn) circle (3pt);
\draw[ultra thin,black] (0,0) -- (\an,\bn);
\tikzmath{\xa = 0; \ya = 0; \xb = 0; \yb = 0;}
\foreach \c in {#6}{
    \ifthenelse{\equal{\c}{R}}{
        \pgfmathparse{\xb+1}
        \xdef \xb {\pgfmathresult}
    }{
        \pgfmathparse{\yb+1}
        \xdef \yb {\pgfmathresult}
    }
    \draw[line width=#2, black] (\xa, \ya) -- (\xb, \yb);
    \ifthenelse{\equal{\c}{R}}{
        \pgfmathparse{\xa+1}
        \xdef \xa {\pgfmathresult}
    }{
        \pgfmathparse{\ya+1}
        \xdef \ya {\pgfmathresult}
    }
}
\foreach \x in {0,...,\an}{
    \node at (\x,-2/5) {\x};
}
\foreach \y in {0,...,\bn}{
    \node at (-1/3,\y) {\y};
}
#8
\end{tikzpicture}
}}

\newcommand{\latticepathstrip}[9]{\resizebox{#1}{!}{
\begin{tikzpicture}
#8
\tikzmath{\a = #3; \b = #4; \n = #5; \an = \a*\n; \i = #6; \im = \i-1; \ani = \an-\i; \bn = \b*\n; \abn = \an+\bn; \nm = \n-1; \abnh = \abn-0.5;}
\foreach \x in {0,...,\ani}{
    \foreach \y in {0,...,\bn}{
        \fill (\x,\y) circle (3pt);
    }
}
\foreach \ii in {0,...,\i}{
    \foreach \j in {1,...,\nm}{
        \fill[white] (\a*\j-\ii,\b*\j) circle (4pt);
    }
}
\fill[fill=gray!30] (0,0) -- (\ani,\ani*\b/\a) -- (\ani,\bn) -- (0,\i*\b/\a) -- cycle;
\fill (0,0) circle (3pt);
\fill (\ani,\bn) circle (3pt);
\draw[ultra thin,black] (0,0) -- (\ani,\ani*\b/\a);
\draw[ultra thin, black] (0,\i*\b/\a) -- (\ani,\bn);
\tikzmath{\xa = 0; \ya = 0; \xb = 0; \yb = 0;}
\foreach \c in {#7}{
    \ifthenelse{\equal{\c}{R}}{
        \pgfmathparse{\xb+1}
        \xdef \xb {\pgfmathresult}
    }{
        \pgfmathparse{\yb+1}
        \xdef \yb {\pgfmathresult}
    }
    \draw[line width=#2, black] (\xa, \ya) -- (\xb, \yb);
    \ifthenelse{\equal{\c}{R}}{
        \pgfmathparse{\xa+1}
        \xdef \xa {\pgfmathresult}
    }{
        \pgfmathparse{\ya+1}
        \xdef \ya {\pgfmathresult}
    }
}
\foreach \x in {0,...,\ani}{
    \node at (\x,-2/5) {\x};
}
\foreach \y in {0,...,\bn}{
    \node at (-1/3,\y) {\y};
}
#9
\end{tikzpicture}
}}

\makeatletter
\newcommand{\citecomment}[2][]{\citen{#2}#1\citevar}
\newcommand{\citeone}[1]{\citecomment{#1}}
\newcommand{\citetwo}[2][]{\citecomment[,~#1]{#2}}
\newcommand{\citevar}{\@ifnextchar\bgroup{;~\citeone}{\@ifnextchar[{;~\citetwo}{]}}}
\newcommand{\citefirst}{\@ifnextchar\bgroup{\citeone}{\@ifnextchar[{\citetwo}{]}}}
\newcommand{\cites}{[\citefirst}
\makeatother

\usepackage{authblk}
\title{Minimal Bridges and a Rotation-Based Bijection}
\author[1]{Benjamin Lou}
\author[2]{Lucas Augustus Brown}
\affil[1]{\begin{small}Center for Theoretical Physics, Massachusetts Institute of Technology, Cambridge, MA 02139, USA; \texttt{benlou5@mit.edu}\end{small}}
\affil[2]{\begin{small}\texttt{lucasbrown.cit@gmail.com}\end{small}}
\date{\the\year--\twodigit{\the\month}--\twodigit{\the\day}}

\begin{document}
\maketitle

\begin{abstract}
A classical problem in lattice path enumeration counts paths that remain on one side of a boundary line. We study several classes of paths where this boundary is porous and show that they are related through a single half-turn rotation bijection. As a first application, we enumerate minimal bridges by relating them to excursions: for positive integers $k$ and $n$, the number of paths from $(0,0)$ to $(kn,n)$ with unit right and up steps that avoid all other lattice points on the line $y=x/k$ is $\frac{k}{kn+n-1}\binom{kn+n}{n}$. The same bijection yields a relation between the ordinary generating functions for binomial coefficients and $k$-Catalan numbers through a dual edge-forbidden model, extends to forbidden strips containing the diagonal, and handles a rational-slope case involving Duchon paths. Finally, our bijection also proves that the number of bridges from $(0,0)$ to $(2n,2n)$ that avoid even diagonal points is $C_{2n}+4C_{2n-1}$, with $C_n$ the $n$th Catalan number. This complements a result of Shapiro.
\end{abstract}

\noindent \textbf{Keywords:} lattice path; bijection; Catalan numbers; ordinary generating function; bridge\\

\noindent \textbf{2020 Mathematics Subject Classification:} Primary 05A15, secondary 05A10, 05A19

\section{Introduction}

By \emph{bridge}, we mean a lattice path that is composed of only unit steps in the rightward and upward directions, starting from $(0,0)$ and ending on a prescribed line through the origin. A classic problem is to count the bridges satisfying a restriction involving the diagonal; for example, bridges staying weakly or strictly above the line are known as \emph{excursions} or \emph{minimal excursions}, respectively.\footnote{It is more standard for the terms ``bridge" and ``excursion" to apply to paths whose distinguished line is the $x$-axis, and whose steps are of the form $(a,b)$ in $\mathbb{Z}\times\mathbb{Z}$, as seen in \cite[Fig. 1]{banderier2002}.  To achieve closed formulas, the steps are usually restricted; with steps of $(1,k)$ and $(1,-1)$, this setting is isomorphic to ours. We consider some generalizations later.}  Examples are provided in Figure \ref{fig:PathTypeExamples1}.  When the distinguished line is $y=x$, excursions are Dyck paths, counted by the Catalan numbers \cite[Prop. 1]{whitworth1878}. More generally, when the line is $y=x/k$ for $k\in \mathbb{Z}^+$, the paths are then counted by the $k$-Catalan numbers \cite{barbier1887,lyness1941}, as shown by Spitzer's cycle lemma \cites[Thm. 2.1]{spitzer1956}[\S10.4]{encylatt}.  In these situations, the diagonal acts as a hard boundary: the path is required to remain on one prescribed side of the line.

We study several ways of making this hard boundary porous, exhibiting a single rotation-based bijection that controls the resulting enumeration in each setting.
In section \ref{results}, we consider \textit{minimal bridges}, these being bridges that do not contain lattice points on the diagonal, aside from the endpoints; see Figure \ref{fig:PathTypeExamples1} for an example. We use our bijection to prove that for $k\in \mathbb{Z}^+$, there are $k+1$ times as many minimal bridges as minimal excursions with endpoint $(kn,n)$; see Figure \ref{fig:square}. 

Specifically, the minimal bridges may be divided into $k+1$ partitions, one of which is exactly the set of minimal excursions. These partitions can be mapped into each other with a suitable rotation, showing that they have the same size. It then follows that the number of minimal bridges $\widetilde{b}_k(n)$ and the corresponding OGF $\widetilde{B}_k(x)$ satisfy
\eqn{\widetilde{b}_k(n)=\frac{k}{kn+n-1}\cdot\binom{kn+n}{n} \qquadtext{and} \widetilde{B}_k(x) \cdot (k + 1 - \widetilde{B}_k(x))^k = (k+1)^{k+1} \cdot x.}
We similarly enumerate a generalization where the path cannot contain any lattice points in an entire strip along the diagonal; see Figure \ref{fig:StripExample1}.

We introduce a condition dual to the minimal condition that we call the \emph{valve condition}. Specifically, a \textit{valved path} can cross the diagonal with an up step but not a right step (the right step off the origin is always allowed). This condition is dual to the minimal condition because it forbids edges rather than lattice points; see Figure \ref{fig:PathTypeExamples1}. We call this the valve condition because the diagonal allows travel in one direction but not the other. With $y=x/k$, our bijection associates each (not-necessarily-minimal) excursion with $k+1$ bridges satisfying this valve condition. Putting these results together completes the square shown in Figure \ref{fig:square}.

Thus far, we have discussed bridges with respect to the line $y=x/k$ for $k\in \mathbb{Z}^+$. More generally, the definition of a bridge allows $k\in \mathbb{Q}^+$, and we consider this possibility in section \ref{generalization}.  Taking the line to be $y=2x/3$, excursions are known as Duchon paths \cite[Example 5]{banderier2002}.  Let $\widetilde{D}(n)$ be the number of minimal Duchon paths, and let $\widetilde{V}(n)$ be the number of minimal valved bridges. Combining our bijection with a
lemma of Nakamigawa and Tokushige \cite[Thm.~1.1]{NakamigawaTokushige2011}, we show that
\eqn{\widetilde{V}(n)=6\widetilde{D}(n) - \frac{4}{5n-1}\cdot\binom{5n-1}{2n}.}

Finally, in section \ref{shapirosection}, we consider bridges ending at $(2n,2n)$, but rather than avoiding all lattice points on the diagonal, we only forbid the diagonal points with even coordinates, as seen in Figure \ref{y22figure}.  Using our rotation bijection and reflection symmetry, we show that there are $C_{2n}+4C_{2n-1}$ such bridges, where $C_k$ is the $k$\textsuperscript{th} Catalan number.  This complements a result of Shapiro \cite[Example 2.5]{shapiro1992}, where he considers the forbidden points to be the diagonal points with odd coordinates: he found that there are $C_{2n}$ such bridges.

\begin{figure}[H]
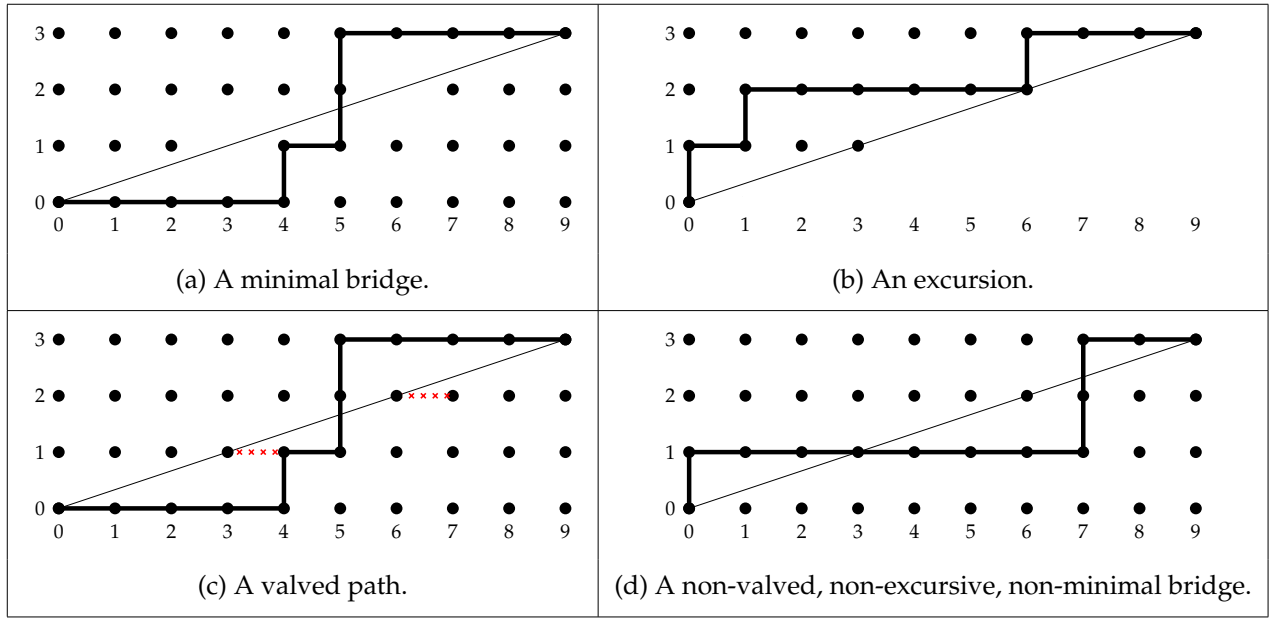

\centering
{\tabulinesep=2mm
\begin{tabu}{|c|c|} \hline
{\latticepath{0.42\textwidth}{0.8mm}{3}{1}{3}{R,R,R,R,U,R,U,U,R,R,R,R}{}{}} & {\latticepath{0.42\textwidth}{0.8mm}{3}{1}{3}{U,R,U,R,R,R,R,R,U,R,R,R}{}{\fill (3,1) circle (3pt); \fill (6,2) circle (3pt); \fill[white] (0.05,-0.2) -- (9.2,-0.2) -- (9.2,2.9) -- cycle;}} \\
(a) A minimal bridge. & (b) An excursion. \\\hline
{\latticepath{0.42\textwidth}{0.8mm}{3}{1}{3}{R,R,R,R,U,R,U,U,R,R,R,R}{}{\draw[thick, decorate, decoration={crosses, segment length = 6pt}, red] (3,1) -- (4,1) (6,2) -- (7,2); \fill (3,1) circle (3pt); \fill (6,2) circle (3pt);}} & {\latticepath{0.42\textwidth}{0.8mm}{3}{1}{3}{U,R,R,R,R,R,R,R,U,U,R,R}{}{\fill (3,1) circle (3pt); \fill (6,2) circle (3pt);}} \\
(c) A valved path. & (d) A non-valved, non-excursive, non-minimal bridge. \\\hline
\end{tabu}}
\caption{Examples of various types of bridge. A bridge is said to be \emph{minimal} if it is nontrivial and avoids all points on the diagonal that are not its endpoints. The reason for this name is that a bridge can be decomposed into sequences of bridges by splitting it at the points where it steps on the diagonal; a minimal bridge is one that is contained in every decomposition of that bridge.  (a) A minimal bridge that also happens to be valved and non-excursive. (b) The excursive criterion excludes all points strictly below the diagonal; therefore, excursions are always valved.  This example also happens to be non-minimal. (c) The valve criterion forbids certain \emph{steps} instead of points.  Both endpoints of a forbidden step are still allowed, unless some other condition excludes them.  This example also happens to be minimal and non-excursive.}
\label{fig:PathTypeExamples1}
\end{figure}

\begin{table}[h]
\centering
\begin{tabular}{|c||c|c|c|c|} \hline
    & \multirow{2}{*}{Bridges} & Minimal & \multirow{2}{*}{Excursions} & Minimal \\
    &         & bridges &            & excursions \\\hline\hline
\rule{0pt}{12pt} The set of all \ldots of arbitrary size & $\mathcal{B}_k$ & $\widetilde{\mathcal{B}}_k$ & $\mathcal{E}_k$ & $\widetilde{\mathcal{E}}_k$ \\\hline
\rule{0pt}{12pt} The set of all \ldots of size $n$ & $\mathcal{B}_k(n)$ & $\widetilde{\mathcal{B}}_k(n)$ & $\mathcal{E}_k(n)$ & $\widetilde{\mathcal{E}}_k(n)$ \\\hline
\rule{0pt}{12pt} Number of \ldots of size $n$ & $b_k(n)$ & $\widetilde{b}_k(n)$ & $e_k(n)$ & $\widetilde{e}_k(n)$ \\\hline
\rule{0pt}{12pt} Ordinary generating function & $B_k$ & $\widetilde{B}_k$ & $E_k$ & $\widetilde{E}_k$ \\\hline
\end{tabular}
\caption{Our symbology.  A bridge from $(0,0)$ to $(kn,n)$ will be said to have \emph{size} $n$.}
\label{SymbolTable}
\end{table}

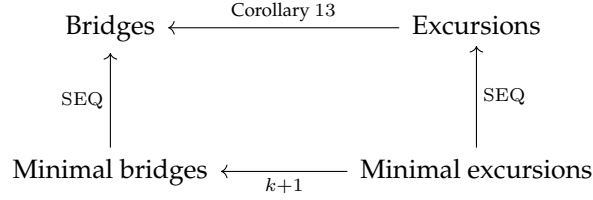
\begin{figure}[H]
\centering
\begin{tikzcd}
\text{Bridges} && \text{Excursions} \\
\\
\text{Minimal bridges} && \text{Minimal excursions}
\arrow["{\text{Corollary }\ref{ben3}}"', from=1-3, to=1-1]
\arrow["{\seq}", from=3-1, to=1-1]
\arrow["{\seq}"', from=3-3, to=1-3]
\arrow["{k+1}", from=3-3, to=3-1]
\end{tikzcd}
\caption{Relationships among the sets, with the arrows pointing in the direction of more paths for a given choice of $k$ and $n$. The SEQ refers to the Sequence operation of the symbolic method \cite[\S I.2.1]{flajolet2009}. Our work establishes a multi-fold bijection linking the left and right sides of the square.} \label{fig:square}\end{figure}

\section{Proofs concerning the case of \texorpdfstring{$y=x/k$}{y=x/k}} \label{results}

We begin with a setting determined by the three integer parameters $k$, $n$, and $i$.  The parameters $k$ and $n$ are positive, and $i$ satisfies $0 \leq i < k$.

\begin{definition}[$d$, the strip, maximum strip, forbidden points]
For a point $(x,y)$, let $d((x,y))$ be the signed horizontal distance from $(x,y)$ to the line $y=x/k$:
\eqn{d((x,y)) = ky - x.}

Let \emph{the strip} be those points $p \not\in \set{(0,0),(kn-i,n)}$ such that $0 \leq d(p) < i$, plus an arbitrary subset of the points on the upper boundary $d(p)=i$.

Any point in the strip is \emph{forbidden}. 

We call the strip \emph{maximum} if it includes all points on the interior of its upper boundary.

We call the strip \emph{minimum} if it includes none of the points on the upper boundary.
\end{definition}

\begin{figure}[H]
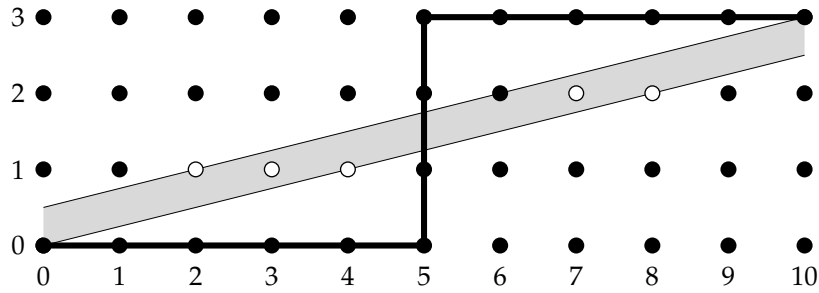

\centering
\latticepathstrip{0.63\textwidth}{0.8mm}{4}{1}{3}{2}{R,R,R,R,R,U,U,U,R,R,R,R,R}{}{\fill (6,2) circle (3pt); \fill (2,1) circle (3pt); \fill[white] (2,1) circle (2.5pt); \fill (3,1) circle (3pt); \fill[white] (3,1) circle (2.5pt); \fill (4,1) circle (3pt); \fill[white] (4,1) circle (2.5pt); \fill (7,2) circle (3pt); \fill[white] (7,2) circle (2.5pt); \fill (8,2) circle (3pt); \fill[white] (8,2) circle (2.5pt);}
\caption{An example of our setting with $(k,n,i)=(4,3,2)$.  The point $(2,1)$ is forbidden and $(6,2)$ is not; therefore, this strip is neither maximum nor minimum.  For this diagram only, forbidden points are indicated by empty circles for clarity.} \label{fig:StripExample1}
\end{figure}

\begin{definition}[Good paths]
A \emph{good path} is a path from $(0,0)$ to $(kn-i,n)$ that avoids all forbidden points.
\end{definition}

The main result of this section is to partition the set of good paths (Definition \ref{PartitionsDef}) and to show that the non-empty partitions have equal sizes by constructing a bijection among them (Theorem \ref{PartitionTheorem}).  We then count the number of good paths in the maximum-strip and minimum-strip cases. We finally extend the logic to count valved paths and derive formulas for related generating functions.

\begin{definition}[Partitions $0$, $1$, \ldots, $k$] \label{PartitionsDef}
For each path, let $p$ be the first lattice point on the path after $(0,0)$ that satisfies $d(p)\geq0$; that is, $p$ is the first point after the start that is weakly above the line $y=x/k$.  Then the path is assigned to Partition $d(p)$.
\end{definition}
For example, the path in Figure \ref{fig:StripExample1} is in Partition 3. In general, a path that is always below the diagonal except for its endpoints will have $p$ be its last point, thus being in Partition $i$, while a path whose first step is upward will have $p = (0,1)$, thus being in Partition $k$. 

These definitions generalize the notions of minimal bridge and minimal excursion.  To see this, note that the maximum strip with $i=0$ is the set of lattice points that are in the interior of the diagonal from $(0,0)$ to $(kn,n)$. Thus, in this case, the set of good paths is the set of minimal bridges $\widetilde{\mathcal{B}}_k$, and Partition $k$ is the set of minimal excursions $\widetilde{\mathcal{E}}_k$; see Figure \ref{fig:i0example}.

\begin{figure}[H]
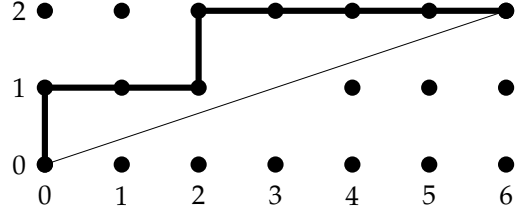

\centering
\latticepath{0.4\textwidth}{0.8mm}{3}{1}{2}{U,R,R,U,R,R,R,R}{}{}
\caption{A good path in Partition $k$, in the case $(k,n,i)=(3,2,0)$ with maximum strip.  This is a minimal bridge and, in fact, a minimal excursion. All paths in Partition $k$ are minimal excursions, because the forbidden points prevent the path from crossing from above the diagonal to below.} \label{fig:i0example}
\end{figure}

\begin{theorem} \label{PartitionTheorem}
For a given choice of $k$, $n$, and $i$, the non-empty partitions are $i$, $i+1$, \ldots, and $k$, and they all have the same size.
\end{theorem}
Note that the theorem holds regardless of which points are included on the upper boundary of the strip.
\begin{proof}
Since an up-step increases the distance $d$ by $k$, we immediately see that only Partitions $0$, $1$, \ldots, $k$ can be non-empty. Furthermore, the strip removes Partitions $0$ through $i-1$.

We now describe the map from Partition $c$ to Partition $k$, with $i \leq c < k$.  For a good path $\ell$ in Partition $c$, let $p(\ell)$ be the first point on the path that is weakly above $y=x/k$, so that $c=d(p(\ell))$. Reversing the sequence of steps up to $p(\ell)$ yields a path in Partition $k$, as shown in Figure \ref{ValvedExample2}.  Geometrically, this reversal can be thought of as a half-turn rotation of the path up to $p(\ell)$.  For the inverse map, let $\ell$ be in Partition $k$; $p(\ell)$ is the first point after the start for which the horizontal distance to $y=x/k$ is $c$. Rotating the path up to $p(\ell)$ returns a path in Partition $c$.

\begin{figure}[H]
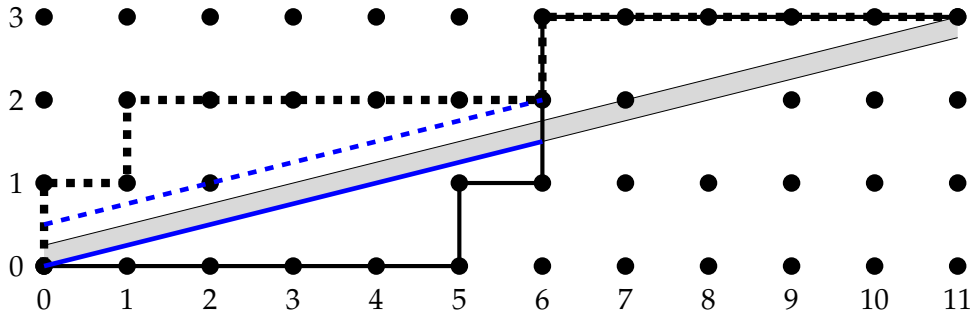

\centering
\latticepathstrip{0.75\textwidth}{0.5mm}{4}{1}{3}{1}{R,R,R,R,R,U,R,U,U,R,R,R,R,R}{}{
    \draw[line width=1mm, dashed] (0,0) -- (0,1) -- (1,1) -- (1,2) -- (6,2) -- (6,3) -- (11,3);
    \draw[ultra thick, dashed, blue] (0,0.5) -- (6,2);
    \draw[ultra thick, blue] (0,0) -- (6,1.5);
    \fill (7,2) circle (3pt);
}
\caption{A good path in Partition $c=2$ (solid) mapped into Partition $k=4$ (dashed).  Here we have $(k,n,i)=(4,3,1)$.  The blue, diagonal solid and dashed segments are the images of each other under the rotation.  The dashed line is at horizontal distance $c=2$ from $y=x/k$; where the original path is below the solid line, the reversed path is above the dashed segment.} \label{ValvedExample2}
\end{figure}

We now justify that this is a bijection between Partitions $c$ and $k$.  For $\ell$ in Partition $c$, all points $q$ before $p(\ell)$ satisfy $d(q) < 0$.  In Figure \ref{ValvedExample2}, these are below the solid blue diagonal line.  After the reversal, $q$ gets sent to a new point $r$ satisfying $d(r) = c - d(q) > c \ge i$.  In Figure \ref{ValvedExample2}, such points are mapped above the dashed diagonal line. Thus, all points up to $p(\ell)$ will be above the strip, so this is indeed a path in Partition $k$.

Conversely, for $\ell$ in Partition $k$, since the distance must decrease from an initial distance of $k$ to a final distance of $i$, and since decrements happen only in steps of $1$, there will exist a point on the path of distance $c$; call the first such point $p(\ell)$. It is easy to see that these maps are inverses of each other.
\end{proof}

\begin{corollary}
The number of good paths is $k+1-i$ times the number of good paths that start with an up-step.
\end{corollary}
\begin{remark}
The $i=0$ case of this corollary can also be obtained using inclusion-exclusion and the identity $$\dbinom{(k+1)j}{j}=\dbinom{(k+1)j-1}{j-1}\cdot(k+1).$$
\end{remark}

We are now ready to count the good paths:

\begin{theorem} \label{goodpaths}
For $0 \leq i < k$, the number of good paths in the maximum-strip case is $$\displaystyle \frac{n(k-i)(k+1-i)}{(kn+n-1-i)(kn+n-i)} \binom{kn+n-i}{n}.$$
\end{theorem}
\begin{proof}
In the maximum-strip case, Partition $i$ is the set of paths that stay below the diagonal until the final step.  Such a path begins with a rightward step and ends with an upward step.  The rest of the path is a translated copy of a path from $(0,0)$ to $(kn-i-1,n-1)$ that is weakly below the line $y=x/k$.  By \cite[Thm. 10.4.5]{encylatt}, Partition $i$ therefore contains
\eqn{\frac{k-i}{kn+n-i-1} \cdot \binom{kn+n-i-1}{n-1}}
paths.  By Theorem \ref{PartitionTheorem}, the number of good paths is $k+1-i$ times the size of Partition $i$, and the result follows immediately.
\end{proof}
\begin{corollary} \label{goodpathsminimum}
In the minimum-strip case, the number of good paths is $\dbinom{kn+n}{n}$ for $i=0$ and
\eqn{\frac{n(k+1-i)^2  }{(kn+n-i) (kn+1-i)} \cdot \binom{kn+n-i}{n}}
for $1 \leq i < k$.
\end{corollary}

\begin{proof}
The case $i=0$ is trivial.  For $i > 0$, we could proceed as in Theorem \ref{goodpaths}, but it is more interesting to understand the relation between the minimum and maximum strip cases. Append a rightward step to the end of the path; the result is a good path in the maximum-strip case for $i-1$, and there are 
\eqn{\frac{n(k-(i-1))(k+1-(i-1))}{(kn+n-1-(i-1))(kn+n-(i-1))} \binom{kn+n-(i-1)}{n}}
such paths. This map is almost a bijection; however, no path in the minimum-strip case with $i$ maps into Partition $i-1$ of the maximum-strip case with $i-1$. To remove this extra partition, we multiply the above quantity by $(k-i+1)/(k-i+2)$.
\end{proof}

\begin{corollary} \label{i0paths}
The number of minimal bridges is $\widetilde{b}_k(n) = \dfrac{k}{kn+n-1} \cdot \dbinom{kn+n}{n}$.
\end{corollary}
\begin{proof}
This follows by taking $i=0$ in Theorem \ref{goodpaths}.
\end{proof}

\begin{remark}
Theorem \ref{PartitionTheorem} and its proof still work if we allow jumps of the form $(kj,j)$.  In this case, it is also interesting to consider adding weights. Specifically, we can assign a number to each type of step and assign each path a weight that is the product of each of its steps. Then the total weight of the paths in each partition will still be the same because the bijection does not change the number of each type of step used. 
\end{remark}

\begin{definition}[Valved path]
\label{def:valve}
Let a \emph{valved path} be a path that does not cross from on or above the diagonal to below the diagonal, except possibly with the first step.  Let $v_k(n)$ be the number of valved paths from $(0,0)$ to $(kn,n)$.
\end{definition}
As seen in Figures \ref{fig:PathTypeExamples1} and \ref{fig6}, this condition is dual to the minimal bridge condition: rather than forbidding the points on the diagonal, we forbid any rightward step that starts on or crosses the diagonal.
\begin{proposition} \label{Prop:ValvedCount}
For $n \geq 1$, $\displaystyle v_k(n) = (k+1) e_k(n) = \frac{k+1}{kn+1}\binom{kn+n}{n}$.
\end{proposition}
\begin{proof}
Using the same bijection (rotate the path up to the first point that is on or above the diagonal), we can show that there are $k+1$ valved paths for every (not-necessarily minimal) excursion.
\end{proof}

We can now complete the square in Figure \ref{fig:square}.  The top edge is completed with the following relation between $B_k$ (the OGF of binomial coefficients $\binom{kn+n}{n}$) and $E_k$ (the OGF of the $k$-Catalan numbers $\frac{1}{kn+1}\binom{kn+n}{n}$): 
\begin{corollary} \label{ben3}
$B_k(x)=\dfrac{E_k(x)}{k+1-k\cdot E_k(x)}$.
\end{corollary}
\begin{proof}
We derive this in two ways.

Let $\mathcal{R}_k(n) = \mathcal{V}_k(n) \setminus \mathcal{E}_k(n)$ be the set of all valved paths $(0,0) \rightarrow (kn,n)$ that are not excursions.  Then by Proposition \ref{Prop:ValvedCount},
\eqn{R_k(x) = k \cdot (E_k(x) - 1).}
The $-1$ removes the trivial path, which is an excursion.

An arbitrary bridge may be decomposed into an excursion followed by a sequence of $\mathcal{R}$-type paths by splitting it at the places where it crosses the diagonal from left to right.  (If the first step of the path is to the right, then the excursion is the trivial path.)  In symbols, $\mathcal{B}_k = \mathcal{E}_k \times \seq(\mathcal{R}_k)$.  This Corollary then follows immediately.\footnote{Recall that in general, if $\mathcal{X}=\seq(\mathcal{Y})$, then their generating functions satisfy $X=1/(1-Y)$.} Note that this decomposition is dual to the decomposition of a bridge into minimal bridges, since we partition by edges rather than points.

We can also derive this result using Corollary \ref{i0paths} and some diagram-chasing in Figure \ref{fig:square}, obtaining the top edge by composing the other three edges (recall our symbology in Table \ref{SymbolTable}): the vertical sides of the diagram are $\mathcal{B}_k = \seq(\widetilde{\mathcal{B}}_k)$ and $\mathcal{E}_k = \seq(\widetilde{\mathcal{E}}_k)$, and our bijection establishes $\widetilde{B}_k(x) = (k+1) \cdot \widetilde{E}_k(x)$.
\end{proof}

\begin{corollary} \label{ben1}
$\widetilde{B}_k(x) \cdot (k+1-\widetilde{B}_k(x))^k = (k+1)^{k+1} \cdot x$.
\end{corollary}
\begin{proof}
By \cite[Theorem 10.4.5, second proof]{encylatt}, we have
$E_k(x) = 1 + x \cdot E_k(x)^{k+1}$. From this and $\mathcal{E}_k = \seq(\widetilde{\mathcal{E}}_k)$, we have $\widetilde{E}_k(x) \cdot (1 - \widetilde{E}_k(x))^k = x$.  The Corollary then follows immediately by $\widetilde{B}_k(x) = (k+1) \cdot \widetilde{E}_k(x)$.
\end{proof}
\section{Valves and minimal Duchon paths} \label{generalization}
So far, we have counted minimal bridges and their generalizations with respect to a line with integer (or reciprocal-integer) slope, meaning $k\in \mathbb{Z}^+$. With $k\in \mathbb{Q}^+$, the proof of Theorem \ref{goodpaths} breaks because it relies on the distance function $d$ taking integer values, whereas here, $d$ can return non-integer rationals. In particular, Partition $k$ no longer maps into all of the other partitions. See Figure \ref{fig6} for an example.

It turns out that we can get the bijection to work just enough to derive a formula relating the number of minimal valved bridges to minimal excursions in the $k=3/2$ case. (Recall the valve condition from Definition \ref{def:valve}.) Minimal excursions with $k=3/2$ will be called \textit{minimal Duchon paths} (\seqnum{A293946}), since they are the minimal version of standard Duchon paths (\seqnum{A060941}).

\begin{figure}[H]
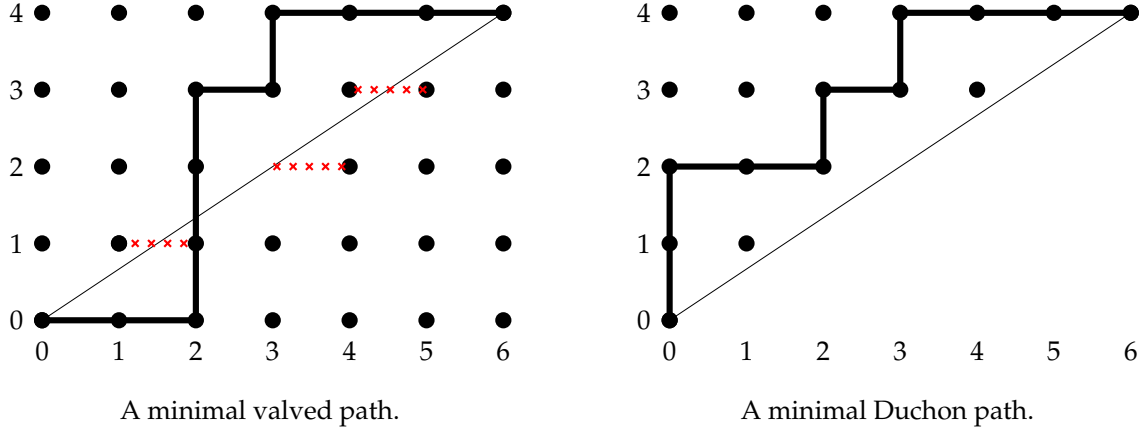

\centering
{\tabulinesep=2mm
\begin{tabu}{cccc}
{\latticepath{0.4\textwidth}{0.8mm}{3}{2}{2}{R,R,U,U,U,R,U,R,R,R}{}{\draw[thick, decorate, decoration={crosses, segment length = 6pt}, red] (1,1) -- (2,1) (3,2) -- (4,2) (4,3) -- (5,3); \fill (1,1) circle (3pt);}} & & & {\latticepath{0.4\textwidth}{0.8mm}{3}{2}{2}{U,U,R,R,U,R,U,R,R,R}{}{\fill[white] (0.1,-0.15) -- (6.2,-0.15) -- (6.2,4) -- cycle;}} \\
A minimal valved path. & & & A minimal Duchon path. \\
\end{tabu}}
\caption{On the left, we have a minimal valved path.  The minimality criterion has the effect of excluding $(3,2)$, while the valve condition is equivalent to forbidding the use of the X-ed edges.  On the right, we have a minimal Duchon path.  It is in Partition $k=3/2$, but it does not map into Partition 1/2 because there are no points on the path whose horizontal distance to the diagonal is 1/2.} \label{fig6}
\end{figure}

We will use the following lemma:

\begin{lemma}[\!\!{\cite[Theorem 1.1]{NakamigawaTokushige2011}}] \label{lemma23}
Consider a line $y=\alpha x$ where $\alpha$ is some real number. Let $P$ be a path starting at the origin and ending at some lattice point $(i,j)$, and assume $P$ stays strictly above the line (excluding the origin). Let $d=j-\alpha i$ be the horizontal distance from the line to the endpoint $(i,j)$. Define a function $\delta(P)$ that returns the minimum horizontal distance between the line and a lattice point on the path (excluding the origin). Summing $\delta$ over all such paths yields
\eqn{\sum_{P} \delta(P) = \frac{d}{i+j} \cdot \binom{i+j}{j}.}
\end{lemma}

Next, let $\widetilde{V}(n)$ and $\widetilde{D}(n)$ be the number of valved minimal bridges and minimal Duchon paths, respectively, from the origin to $(3n,2n)$. 
\begin{theorem}
The number of valved minimal bridges is $\displaystyle \widetilde{V}(n)=6\widetilde{D}(n) - \frac{4}{5n-1}\cdot\binom{5n-1}{2n}$.
\end{theorem}

\begin{proof}
We begin by observing that a half-turn rotation of the entire path is an involution on the set of valved minimal bridges.  The forbidden edges that cross the diagonal are simply exchanged under the half-turn, while the forbidden edges that have an endpoint on the diagonal are already banned by the minimality condition, which is again symmetric under the half-turn rotation.\footnote{This latter type of edge means that the set of all valved bridges does \textit{not} have a half-turn rotation symmetry.}

We split the set of valved minimal bridges into four partitions as in Definition \ref{PartitionsDef}: $0$, $1/2$, $1$, and $3/2$. The half-turn rotation of the entire path constitutes a bijection between Partitions $0$ and $3/2$, these being the paths strictly above and below the diagonal, respectively. Thus, each has $\widetilde{D}(n)$ paths. The same map gives a bijection between Partitions $1/2$ and $1$; define $D_1(n)$ as the number of paths in either partition. Then the number of minimal valved bridges is
\begin{equation}
\label{Mn from partitions}
\widetilde{V}(n)=2\widetilde{D}(n) + 2D_1(n).
\end{equation}
Lemma \ref{lemma23} lets us express $D_1(n)$ in terms of $\widetilde{D}(n)$. Take the endpoint in the Lemma as $(3n-1,2n)$. For a path $P$ in the Lemma, we have that $\delta(P)$ is either $1/2$ or $1$. Let $\mathcal{A}_{1/2}$ be the set of $P$ with $\delta(P)=1/2$, and similarly $\mathcal{A}_1$ for the set with $\delta(P)=1$. We will now show that $|\mathcal{A}_{1/2}|+|\mathcal{A}_1|=\widetilde{D}(n)$ and $|\mathcal{A}_{1/2}|=D_1(n)$.

Let $f(P)$ be $P$ with an extra rightward step appended at the end. Then $f$ is a bijection between $\mathcal{A}_{1/2} \cup \mathcal{A}_1$ and Partition $3/2$, yielding the first equality. Furthermore, $f(\mathcal{A}_{1/2})$ can be mapped to Partition $1/2$ by our bijection, rotating each path in $f(\mathcal{A}_{1/2})$ up to its first point of distance 1/2. This proves the second equality. Combined with Lemma \ref{lemma23}, we find
\eqn{\frac{1}{2} \abs{\mathcal{A}_{1/2}} + \abs{\mathcal{A}_1} = \frac{1}{2} \cdot D_1(n) + (\widetilde{D}(n) - D_1(n)) = \frac{1}{5n-1} \cdot \binom{5n-1}{2n}.}
Substituting into \ref{Mn from partitions} to eliminate $D_1$ yields the result.
\end{proof}

\section{Bridges with a subset of the diagonal forbidden} \label{shapirosection}

Let $a$, $b$, and $n$ be positive integers, and consider bridges from $(0,0)$ to $(an,bn)$ where the forbidden points are $(ai,bi)$ with $0<i<n$.  The case with $\gcd(a,b)=1$ yields a minimal bridge, but when $\gcd(a,b)\neq1$, only a subset of the points on the diagonal are forbidden.

Let $\mathcal{Y}_{a,b}(n)$ denote the set of paths in this new setting, and let its cardinality and OGF be $y_{a,b}(n)$ and $Y_{a,b}$, respectively.  The observation $\mathcal{B}_k = \seq(\widetilde{\mathcal{B}}_k)$ generalizes directly to this new setting,\footnote{It is interesting to observe that the same bridge can be built as a sequence of different versions of the paths described in this section. The bridges from $(0,0)$ to $(2n,2n)$ can be built out of minimal bridges (that is, bridges from $\mathcal{Y}_{1,1}$) or bridges from $\mathcal{Y}_{2,2}$; however, $\seq(\mathcal{Y}_{1,1})$ contains all bridges while $\seq(\mathcal{Y}_{2,2})$ contains only even-length bridges.} yielding
\neqn{Y_{a,b}(x) = 1 - \left( \sum_{n=0}^\infty \binom{an+bn}{bn} \cdot x^n \right)^{-1}, \label{Yab}}
which can be used to numerically calculate the elements of the sequence. This further implies
\neqn{\binom{(a+b)\cdot i}{b\cdot i}=\sum_{j=1}^i y_{a,b}(j)\cdot\binom{(a+b)\cdot(i-j)}{b\cdot(i-j)}, \label{yab}}
which also immediately follows from the observation that any bridge built from $\mathcal{Y}_{a,b}(n)$ can be decomposed into an initial piece counted by $y_{a,b}(j)$ and the rest, which is a smaller such bridge. A similar observation holds between any combinatorial class and its sequence.\footnote{Line (\ref{yab}) can also be derived by invoking the Lindstr\"{o}m--Gessel--Viennot lemma \cite{dpapp,lindstrom}, which in a special case counts lattice paths avoiding a prescribed set of points \cite[Lemma 10.7.2]{encylatt} \cite[Ex.\ 2.7.2]{stanley2012}.}

In the rest of this section, we specialize to the case of $a=b=2$; numerical investigations of other cases may be found in Appendix \ref{app:numerics}.
Let $C_n$ be the $n$\textsuperscript{th} Catalan number.  In \cite{shapiro1992}, Shapiro considered bridges to the line $y=x$ that avoid all points on the diagonal whose coordinates are odd; he concluded \cite[Example 2.5]{shapiro1992} that the number of such bridges from $(0,0)$ to $(2n,2n)$ is $C_{2n}$.  A complementary problem is to count the bridges between those points that avoid the diagonal points with \emph{even} coordinates; that is, to determine $y_{2,2}(n)$.  

\begin{figure}[H]
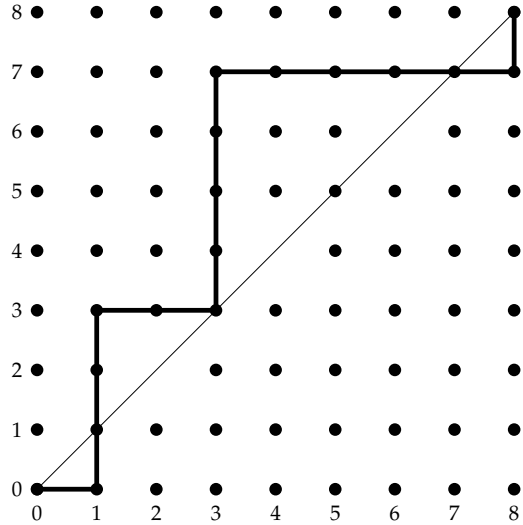

\centering
\latticepath{0.4\textwidth}{0.8mm}{2}{2}{4}{R,U,U,U,R,R,U,U,U,U,R,R,R,R,R,U}{}{}
\caption{A bridge in $\mathcal{Y}_{2,2}(4)$.} \label{y22figure}
\end{figure}

Note that the forbidden points of the $a=b=2$ situation may be seen as a nonmaximal strip with $i=0$; however, Theorem \ref{PartitionTheorem} does not provide much information in this case, since both partitions are just as hard to count as $y_{2,2}(n)$. A different approach is thus needed.  The result turns out to be 
\eqn{y_{2,2}(n) = C_{2n} + 4C_{2n-1} = \frac{8n + 1}{8n^2 + 2n - 1} \cdot \binom{4n}{2n},}
which we added to the On-Line Encyclopedia of Integer Sequences as $\seqnum{A337350}$ \cite{oeis}.  Two analytic proofs were initially produced in a collaboration with users on the Mathematics Stack Exchange \cite{nhan2025,user2025} by using (\ref{yab}) as a starting point, but we later found two combinatorial proofs: one is bijective, using both rotation and reflection symmetry, and the other uses the symbolic method and generating functions.  We present the bijective proof first.

To organize the proof, we introduce the following framework. Fix $n$ and choose three sets of lattice points: the initial forbidden points $F_1$, the target points $T$, and the final forbidden points $F_2$. Let $\mathcal{K}(F_1, T, F_2)$ be the set of paths from $(0,0)$ to $(2n,2n)$ that avoid $F_1$ until they hit $T$ for the first time, and then avoid $F_2$ strictly after that first hit. We also demand that such a path contains at least one point in $T$ unless $T=\emptyset$. We will write $K(F_1, T, F_2)=\abs{\mathcal{K}(F_1, T, F_2)}$. Let $E$ be the set of even diagonal points, $O$ the odd diagonal points, and $D = E \cup O$ all of the diagonal points, where all of these sets exclude endpoints. Then Shapiro showed \cite[Example 2.5]{shapiro1992} that $K(\emptyset, \emptyset, O)=C_{2n}$, and we want to find $y_{2,2}(n)=K(\emptyset, \emptyset, E)$.

We have the following lemmas.

\begin{figure}[H]
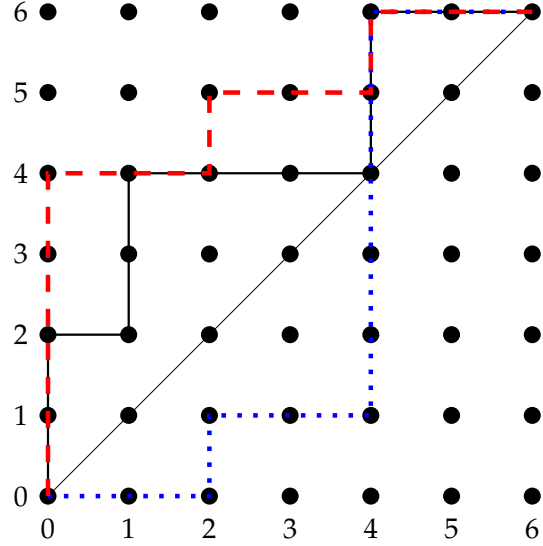

\centering
\latticepath{0.42\textwidth}{0.8mm}{6}{6}{1}{}{}{
\draw[thick] (0,0) -- (0,2) -- (1,2) -- (1,4) -- (4,4) -- (4,6) -- (6,6);
\draw[loosely dotted, ultra thick, blue] (0,0) -- (2,0) -- (2,1) -- (4,1) -- (4,6) -- (6,6);
\draw[dash pattern=on6off6, ultra thick, red] (0,0) -- (0,4) -- (2,4) -- (2,5) -- (4,5) -- (4,6) -- (6,6);
}
\caption{An example of the three-way bijection from Lemma \ref{lemma:KnullDO}, with $n=4$.  The solid path is in $\mathcal{K}(\emptyset,D,O) \setminus S$, the blue dotted path is the corresponding path in $\mathcal{S}$, and the red dashed path is the corresponding path in $\mathcal{K}(\emptyset,\emptyset,O)$.} \label{fig:Kbijections}
\end{figure}

\begin{lemma} \label{lemma:KnullDO}
$K(\emptyset,D,O) = 2C_{2n}$.
\end{lemma}
\begin{proof}
Let $\mathcal{S} \subset \mathcal{K}(\emptyset, D, O)$ be the subset of paths that cross the diagonal upon their first meeting with the diagonal. In other words, the steps adjacent to the first crossing are RR or UU. Figure \ref{fig:Kbijections} contains some examples of these sets.

We claim that the following three sets are in bijection: $\mathcal{S}$,  $\mathcal{K}(\emptyset,D,O)\setminus\mathcal{S}$, and the Shapiro paths $\mathcal{K}(\emptyset, \emptyset, O)$.  Shapiro's result then shows that each of these sets has $C_{2n}$ elements, which suffices to prove the lemma since $\mathcal{K}(\emptyset, D, O) = \mathcal{S}\sqcup (\mathcal{K}(\emptyset, D, O) \setminus \mathcal{S}).$ For the first two sets, both the bijection and its inverse reflect the path about the diagonal up to the first diagonal point. Note that this is a reflection about the diagonal, not a half-turn rotation.  See Figure \ref{fig:Kbijections} for an example. 

Meanwhile, $\mathcal{S}$ and the Shapiro paths $\mathcal{K}(\emptyset, \emptyset, O)$ are related by our rotation bijection.  To map $\mathcal{S}$ into  $\mathcal{K}(\emptyset,\emptyset,O)$, let $p$ be the first point on the path after it meets the diagonal for the first time, and rotate the path up to $p$ by a half-turn; notice that this eliminates the intersection with the diagonal before $p$.  For the inverse map, let $p$ be the \textit{second} point the path in $\mathcal{K}(\emptyset, \emptyset, O)$ has a horizontal distance of $\pm 1$ from the diagonal, and rotate the path up to $p$.  See Figure \ref{fig:Kbijections} for an example. Note that the first point where this distance is achieved will always be $(0,1)$ or $(1,0)$. This point becomes the diagonal crossing after the rotation, and since the first two steps of a path in $\mathcal{K}(\emptyset, \emptyset, O)$ are always in the same direction, the rotated path will indeed be in $\mathcal{S}$. 
\end{proof}

\begin{lemma} \label{lemma:KEOO}
$K(E,O,O) = C_{2n} + 2C_{2n-1}$.
\end{lemma}
\begin{proof}
The paths in $\mathcal{K}(\emptyset, D, O)$ can be partitioned into two subsets based on whether their first diagonal crossing is an even or an odd lattice point; therefore,
\eqn{K(\emptyset,D,O) = K(E,O,O) + K(O,E,O).}
The LHS is $2C_{2n}$ by Lemma \ref{lemma:KnullDO}. Therefore, it suffices to show that
\neqn{K(O,E,O) = C_{2n} - 2C_{2n-1}. \label{opaujnertfb}}
The set $\mathcal{K}(O,E,O)$ is the same as $\mathcal{K}(\emptyset,\emptyset,O)$ but without those paths that do not touch the diagonal. Therefore,
\eqn{K(O,E,O) = K(\emptyset,\emptyset,O) - K(\emptyset,\emptyset,D).}
Now $\mathcal{K}(\emptyset,\emptyset,O)$ is the set of Shapiro paths. Meanwhile, $\mathcal{K}(\emptyset,\emptyset,D)$ is the set of paths that stay strictly above or below the diagonal, so\footnote{The paths strictly below the diagonal are the Dyck paths with $2n-1$ rightward steps, plus an initial rightward step and a final upward step, and the paths above the diagonal are the half-turn rotations thereof.}
\neqn{\mathcal{K}(\emptyset,\emptyset,D)= 2C_{2n-1}. \label{non-diag}}
This proves \ref{opaujnertfb} and hence the result.
\end{proof}

\begin{theorem} \label{y22thm}
$y_{2,2}(n) = C_{2n} + 4C_{2n-1}$.
\end{theorem}
\begin{proof}
First, recall that $y_{2,2}(n) = K(\emptyset,\emptyset,E)$.  The set $\mathcal{K}(E,O,E)$ is the same as $\mathcal{K}(\emptyset,\emptyset,E)$, except that the latter contains paths that do not touch the diagonal.  Therefore,
\neqn{K(\emptyset,\emptyset,E) = K(E,O,E) + K(\emptyset,\emptyset,D). \label{kjnarbfkujn}}
A bijection exists between $\mathcal{K}(E,O,O)$ and $\mathcal{K}(E,O,E)$: in both directions, perform a half-turn rotation on the part of the path after the first diagonal point.  (See Figure \ref{fig:KEOO} for an example.) This is a bijection because the rotation effectively exchanges the even and odd forbidden points. Therefore, $K(E,O,E) = K(E,O,O)$.  We computed $K(E,O,O)$ in Lemma \ref{lemma:KEOO}, and $K(\emptyset,\emptyset,D)=2C_{2n-1}$ (see (\ref{non-diag})), so the result follows upon substitution into (\ref{kjnarbfkujn}).
\end{proof}

\begin{figure}[H]
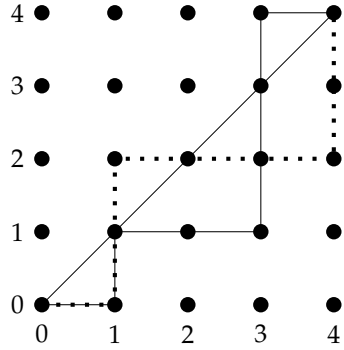

\centering
\latticepath{0.27\textwidth}{0.8mm}{4}{4}{1}{}{}{
\draw (0,0) -- (1,0) -- (1,1) -- (3,1) -- (3,4) -- (4,4);
\draw[loosely dotted, ultra thick] (0,0) -- (1,0) -- (1,1) -- (1,2) -- (4,2) -- (4,4);
}
\caption{A path in $\mathcal{K}(E,O,O)$ (dotted) and the corresponding path in $\mathcal{K}(E,O,E)$ (solid).} \label{fig:KEOO}
\end{figure}
This concludes the bijective proof.  For our generating-function proof, we decompose the paths into subpaths, use the symbolic method to produce a generating function in terms of the Catalan OGF, and simplify the result.

\begin{proof}[Analytic proof of Theorem \ref{y22thm}]
Let an \emph{arch} be a minimal bridge with respect to the line $y=x$.\footnote{In the previous proof, we would have called this an element of $\mathcal{K}(\emptyset,\emptyset,D)$.} In this situation, we say that the \emph{length} of an arch is the number of rightward steps that it contains. For example, if the arch is from $(0,0)$ to $(m,m)$, then its length is $m$.

The key decomposition is that each path in $\mathcal{Y}_{2,2}(n)$ is either a single even-length arch, or two odd-length arches bookending a (possibly-empty) sequence of even-length arches. It remains to write down the generating function for the even- and odd-length arches. Firstly, let $C(x)$ be the OGF of the Catalan numbers, and let $C_E(x)$ and $C_O(x)$ be its even and odd parts, respectively:
\eqn{C(x) = \sum_{n=0}^\infty \frac{1}{n+1} \binom{2n}{n} \cdot x^n = \frac{1-\sqrt{1-4x}}{2x}}
\eqn{C_E(x) = \sum_{n=0}^\infty \frac{1}{2n+1} \binom{4n}{2n} \cdot x^{2n} = \frac{C(x) + C(-x)}{2} = \frac{\sqrt{1+4x} - \sqrt{1-4x}}{4x}}
\eqn{C_O(x) = \sum_{n=0}^\infty \frac{1}{2n+2} \binom{4n+2}{2n+1} \cdot x^{2n+1} = \frac{C(x) - C(-x)}{2} = \frac{2 - \sqrt{1+4x} - \sqrt{1-4x}}{4x}}
Now consider the generating function $Y_{2,2}(x)=\sum_n y_{2,2}(n) x^{n}$.  Because there are $2C_{m-1}$ arches of length $m$, the key decomposition yields
\neqn{Y_{2,2}(x) = \left( 2\sqrt{x} C_E(\sqrt{x}) \right)^2 \cdot \frac{1}{1 - 2\sqrt{x} C_O(\sqrt{x})} + 2\sqrt{x} C_O(\sqrt{x}). \label{Y22x}}
The rest of the proof amounts to plugging in the above expressions for $C_E(x)$ and $C_O(x)$ and carefully simplifying. The details can be found in Appendix \ref{app:alg}.
\end{proof}
\begin{corollary} \label{y22cor}
$Y_{2,2}(x) = 1 - C_E(\sqrt{x}) \sqrt{1-16x}= 1 - \dfrac{\sqrt{2}\sqrt{1-16x}}{\sqrt{1+\sqrt{1-16x}}}$.
\end{corollary}
The proof can be found in Appendix \ref{app:alg}.

\section{Acknowledgements}

We would like to thank William Kuszmaul and David W.\ Brown for their advice on preparing this paper, and Qu\'{y} Nh\^{a}n and StackExchange user ``user" for a different proof of Theorem (\ref{y22thm}) (not reproduced here). We also thank an anonymous reviewer from the Journal of Integer Sequences for comments on an early draft. B. Lou would also like to thank Holden Mui for a discussion about the proper definition of ``minimal". B. Lou also gratefully acknowledges support from the  Fannie and John Hertz Foundation and from a first-year graduate fellowship through the MIT Department of Physics and the MIT Center for Theoretical Physics.


\appendix

\section{Numerical checks for further formulas}\label{app:numerics}
Based on the results of this paper---Corollary \ref{i0paths} and Theorem \ref{y22thm} in particular---one might conjecture that
\neqn{y_{a,b}(n)\bigg/\dbinom{(a+b)\cdot n}{b\cdot n} \label{ybinom}}
is a rational function of $a$, $b$, and $n$; however, further computation suggests that this conjecture holds only when $a=b=2$ or when $1\in\{a,b\}$.  For example, when $b=3$ and $a\in\{2,3\}$, we have the following data for $1\leq n\leq9$:
\eqn{\begin{array}{|c|c|c|c|c|}\hline
n & \begin{array}{c}y_{2,3}(n)\\\seqnum{A337351}(n)\end{array} & y_{2,3}(n)\bigg/\dbinom{5n}{2n} & \begin{array}{c}y_{3,3}(n)\\\seqnum{A337352}(n)\end{array} & y_{3,3}(n)\bigg/\dbinom{6n}{3n} \\\hline
1 &           10 &                               1 &              20 &                               1 \\\hline
2 &          110 &                           11/21 &             524 &                         131/231 \\\hline
3 &         1805 &                        361/1001 &           19660 &                        983/2431 \\\hline
4 &        34770 &                          61/221 &          854380 &                   213595/676039 \\\hline
5 &       731760 &                     18294/81719 &        40304080 &                  503801/1938969 \\\hline
6 &     16295600 &                  651824/3459729 &      2004409236 &             167034103/756261275 \\\hline
7 &    377438250 &                   148015/909788 &    103440770760 &            862006423/4485482287 \\\hline
8 &   8999246900 &            179984938/1257042033 &   5486614131756 &      457217844313/2687300306925 \\\hline
9 & 219399101415 &        43879820283/343176898988 & 297239307415792 &      379131769663/2483341104143 \\\hline
\end{array}}
The number of digits in these fractions appears to increase roughly linearly with $n$, whereas if (\ref{ybinom}) were a rational function of $n$, then we would expect the number of digits to increase logarithmically.

\section{Generating-function proof of Theorem \ref{y22thm}}\label{app:alg}
In line (\ref{Y22x}), we had
\eqn{Y_{2,2}(x) = \left( 2\sqrt{x} C_E(\sqrt{x}) \right)^2 \cdot \frac{1}{1 - 2\sqrt{x} C_O(\sqrt{x})} + 2\sqrt{x} C_O(\sqrt{x}),}
where
\eqn{C(x) = \frac{1-\sqrt{1-4x}}{2x},}
\neqn{C_E(x) = \frac{C(x)+C(-x)}{2} = \frac{\sqrt{1+4x} - \sqrt{1-4x}}{4x}, \label{CEfmla}}
and
\neqn{C_O(x) = \frac{C(x)-C(-x)}{2} = \frac{2 - \sqrt{1+4x} - \sqrt{1-4x}}{4x}. \label{COflma}}
To suppress the proliferation of square roots and halved exponents, we replace $x$ with $x^2$:
\neqn{Y_{2,2}(x^2) = \frac{4x^2 \cdot C_E(x)^2}{1 - 2x \cdot C_O(x)} + 2x \cdot C_O(x). \label{vpzobxui}}
Now the Catalan OGF satisfies $C(x) = 1 + x \cdot C(x)^2$, so
\eqn{C(x) = 1 + x \cdot (C_E(x) + C_O(x))^2}
\eqn{ = 1 + x C_E(x)^2 + x C_O(x)^2 + 2 x C_E(x) C_O(x).}
The middle two terms are both odd, while the outer two terms are both even.  Therefore
\neqn{C_E(x) = 1 + 2 x C_E(x) C_O(x). \label{CECO}}
Substituting this into (\ref{vpzobxui}) to eliminate the denominator then yields
\eqn{Y_{2,2}(x^2) = 4x^2 \cdot C_E(x)^3 + 2x \cdot C_O(x).}
Substituting from (\ref{CEfmla}) and (\ref{COflma}) then yields
\neqn{Y_{2,2}(x^2) = 1 + \frac{(1-4x) \sqrt{1+4x} - (1+4x) \sqrt{1-4x}}{4x}. \label{Y22x2}}
Invoking the binomial series yields
\eqn{ = 1 + \frac{1}{4x} \left( (1-4x) \sum_{n=0}^\infty \binom{2n}{n} \frac{(-1)^{n+1}}{2n-1} x^n - (1+4x) \sum_{n=0}^\infty \binom{2n}{n} \frac{(-1)^{n+1}}{2n-1} (-x)^n \right),}
which we then combine into a single sum to obtain
\eqn{ = 1 + \frac{1}{4x} \sum_{n=0}^\infty \left( (1+4x) - (1-4x) (-1)^n\right) \binom{2n}{n} \frac{x^n}{2n-1}.}
We then split this sum into its even- and odd-indexed subseries as
\eqn{ = 1 + \frac{1}{4x} \left( \sum_{\substack{n=0 \\ n \text{ odd}}}^\infty 2 \binom{2n}{n} \frac{x^n}{2n-1} + \sum_{\substack{n=0 \\ n \text{ even}}}^\infty 8x \binom{2n}{n} \frac{x^n}{2n-1} \right),}
which we reindex as
\eqn{ = 1 + \frac{1}{4x} \left( \sum_{n=0}^\infty \binom{4n+2}{2n+1} \frac{2x^{2n+1}}{4n+1} + \sum_{n=0}^\infty \binom{4n}{2n} \frac{8x^{2n+1}}{4n-1} \right).}
Combining these sums back into a single series and simplifying finally yields
\eqn{Y_{2,2}(x^2) = \sum_{n=1}^\infty \frac{8n+1}{8n^2+2n-1} \binom{4n}{2n} x^{2n}.}
This completes the proof of Theorem \ref{y22thm}. To verify Corollary \ref{y22cor}, let
\eqn{A = \sqrt{1+4x} \qquadtext{and} B = \sqrt{1-4x}}
so that
\eqn{AB = \sqrt{1-16x^2} \qquadtext{and} C_E(x)=\frac{A-B}{4x}.}
The second equality is (\ref{CEfmla}). Therefore, using (\ref{Y22x2}),
\neqn{Y_{2,2}(x^2) = 1-\frac{A^2B-AB^2}{4x} = 1-C_E(x)\,AB = 1-C_E(x)\sqrt{1-16x^2}. \label{Y22x_using_CE}}
This yields the first equality in Corollary 4.1.  For the second equality, rationalize $C_E(x)$:
\eqn{
C_E(x) = \frac{A-B}{4x} \cdot \frac{A+B}{A+B} = \frac{A^2-B^2}{4x(A+B)} = \frac{8x}{4x(A+B)} = \frac{2}{A+B}.}
Moreover,
\eqn{(A+B)^2 = A^2 + B^2 + 2AB = (1+4x) + (1-4x) + 2\sqrt{1-16x^2} = 2 \left(1 + \sqrt{1-16x^2} \right).}
Upon taking the square root, we observe that the minus root is extraneous; therefore,
\eqn{A+B = \sqrt{2} \, \sqrt{1+\sqrt{1-16x^2}},}
so
\eqn{C_E(x) = \frac{\sqrt{2}}{\sqrt{1+\sqrt{1-16x^2}}}.}
Substituting into (\ref{Y22x_using_CE}) yields
\eqn{Y_{2,2}(x^2) = 1 - \frac{\sqrt{2}\,\sqrt{1-16x^2}}{\sqrt{1+\sqrt{1-16x^2}}},}
which is the claimed closed form.

\end{document}